\documentclass{article}
\usepackage[utf8]{inputenc}
\usepackage[margin=1 in]{geometry}
\usepackage{amsmath}
\usepackage{amsfonts}
\usepackage{amssymb}
\usepackage{amsthm}
\usepackage{mathrsfs}
\usepackage{enumitem}
\usepackage{hyperref}
\usepackage{xcolor}

\newtheorem{theorem}{Theorem}
\numberwithin{theorem}{section}
\newtheorem{lemma}[theorem]{Lemma}

\newtheorem{claim}[theorem]{Claim}

\theoremstyle{definition}
\newtheorem{definition}[theorem]{Definition}
\newtheorem{example}[theorem]{Example}

\newcommand{\tb}{\textbf}
\newcommand{\tn}{\textnormal}
\newcommand{\se}{\subseteq}
\newcommand{\til}{\widetilde}
\newcommand{\ol}{\overline}
\newcommand{\lam}{\lambda}
\newcommand{\mf}{\mathfrak}
\newcommand{\tcb}{\textcolor{blue}}

\newcommand{\p}{\ol{\til{\pi}}}
\renewcommand{\P}{\ol{\mf{P}}}
\newcommand{\thet}{\ol{\til{\theta}}}
\newcommand{\F}{\ol{\mf{F}}}
\newcommand{\del}{\til{\partial}}

\usepackage[maxbibnames=10,style=alphabetic]{biblatex}
\title{$K$-analogues of Hivert's divided difference operators}
\author{Laura Pierson \\ \href{mailto:lcpierson73@gmail.com}{lcpierson73@gmail.com}}
\begin{document}

\maketitle

\begin{abstract}
    Several families of polynomials of combinatorial and representation theoretic interest (notably the Schur polynomials $s_\lam$, Demazure characters $\mf{D}_a$, and Demazure atoms $\mf{A}_a$) can be defined in terms of divided difference operators. Hivert \cite{hivert2000hecke} defines ``fundamental analogues" of these divided difference operators, and Hivert and Hicks-Niese \cite{hicks2024quasisymmetric} show that the polynomials that arise from those fundamental operators in analogous ways to the three families of polynomials above are respectively the fundamental quasisymmetric functions $F_a$ from \cite{gessel1984multipartite}, the fundamental slides $\mf{F}_a$ from \cite{assaf2017schubert}, and the fundamental particles $\mf{P}_a$ from \cite{searles2020polynomial}. Lascoux \cite{lascoux2001transition} defines $K$-analogues of the divided difference operators, and Buciumas, Scrimshaw, and Weber \cite{buciumas2020colored}  show that the polynomials arising in corresponding ways from the $K$-theoretic divided difference operators are respectively the Grothendieck polynomials $\ol{s}_\lam$, the combinatorial Lascoux polynomials $\ol{\mf{D}}_a$ from \cite{monical2016set}, and the combinatorial Lascoux atoms $\ol{\mf{A}}_a$ from \cite{monical2016set}, as conjectured by Monical \cite{monical2016set}. We define $K$-analogues of Hivert's fundamental divided difference operators and show that the polynomials arising in the corresponding ways from our new operators are respectively the multifundamentals $\ol{F}_a$ from \cite{lam2007combinatorial}, the fundamental glides $\F_a$ from \cite{pechenik2019decompositions}, and the kaons $\ol{\mf{P}}_a$ from \cite{monical2021polynomials}.
\end{abstract}

\section{Introduction}

Several classic families of polynomials that are of combinatorial and representation theoretic interest can be defined in terms of the divided difference operators $$\partial_i := \frac{1-s_i}{x_i - x_{i+1}},$$ where $s_i$ is the involution on polynomials that swaps the variables $x_i$ and $x_{i+1}$. Specifically, let $\pi_i := \delta_i x_i$ and $\theta_i := \pi_i-1 = x_{i+1}\delta_i,$ and for a permutation $w$ and a reduced word $\sigma_{i_1}\dots\sigma_{i_\ell}$ for $w$ where $\sigma_i = (i \ i+1)$, define $\pi_w := \pi_{i_1}\dots \pi_{i_\ell}$ and $\theta_w := \theta_{i_1}\dots\theta_{i_\ell}.$ Then we have:

\begin{theorem}[Demazure \cite{demazure1974nouvelle}]\label{thm:classic}
    If $\lambda\vdash n$ is a partition, $w_0\in\mf{S}_n$ is the longest permutation, $a$ is a strong composition, $\tn{sort}(a)$ is the partition formed by rearranging the parts of $a$ in nondecreasing order, and $w_a$ is the minimal length permutation turning $a$ into $\tn{sort}(a)$, we have $$\pi_{w_0}(x^\lambda) = s_\lambda,\hspace{1cm}\pi_{w_a^{-1}}(x^{\tn{sort}(a)}) = \mf{D}_a,\hspace{1cm}\theta_{w_a^{-1}}(x^{\tn{sort}(a)}) = \mf{A}_a,$$ where $s_\lambda$ are the Schur polynomials, $\mf{D}_a$ are the Demazure characters from \cite{demazure1974nouvelle} (also called key polynomials and denoted $\kappa_a$), and $\mf{A}_a$ are the Demazure atoms from \cite{demazure1974nouvelle} (denoted $\mathcal{A}_a$ in \cite{hicks2024quasisymmetric}).
\end{theorem}

The polynomials in Theorem \ref{thm:classic} have combinatorial tableau definitions which we will not repeat here because we will not need them, and they also have representation theoretic interpretations given in \cite{demazure1974nouvelle}. 

For $w\in \mf{S}_n$, the Schubert polynomials $\mf{S}_w$ also have a definition $\mf{S}_w = \partial_{w^{-1}w_0}(x_1^{n-2}x_2^{n-2}\dots x_{n-1}^1)$ in terms of divided difference operators. Schubert polynomials (and Schur polynomials, which are a special case) also have a geometric meaning related to the cohomology rings of flag varieties. Varieties also have a different ring structure associated to them called the $K$-ring, generated by line bundles, with addition and multiplication defined in terms of direct sums and tensor products. The $K$-ring depends on an additional parameter $\beta$ and specializes to the cohomology ring when $\beta = 0$. The $K$-rings of flag varieties give rise to the Grothendieck polynomials $\ol{\mf{S}}_w$, which are thus the $K$-analogue of the Schubert polynomials. In the Schur polynomial case, these $K$-analogues $\ol{s}_\lam$ have a nice combinatorial description in terms of set-valued Young tableaux that lifts the tableau definition of Schur polynomials. Motivated by that, a number of authors have defined $K$-analogues of various combinatorially defined polynomials by replacing the relevant combinatorial objects with corresponding ``set-valued" objects, in the hopes that the original polynomials and their $K$-analogues may have some geometric meaning. 

In that vein, Lascoux \cite{lascoux2001transition} defines $K$-analogues $\ol{\pi}$ and $\ol{\theta}$ of the above divided difference operators, and Lascoux \cite{lascoux2001transition} and Buciumas, Scrimsaw, and Weber \cite{buciumas2020colored} prove a $K$-analogue of Theorem \ref{thm:classic}:

\begin{theorem}[Lascoux \cite{lascoux2001transition}; Buciumas-Scrimshaw-Weber \cite{buciumas2020colored}]\label{thm:K}
    With $\lam$, $a$, $\tn{sort}(a)$, and $w_a$ as in Theorem \ref{thm:classic}, $$\ol{\pi}_{w_0}(x^\lambda) = \ol{s}_\lambda,\hspace{1cm}\ol{\pi}_{w_a^{-1}}(x^{\tn{sort}(a)}) = \ol{\mf{D}}_a,\hspace{1cm}\ol{\theta}_{w_a^{-1}}(x^{\tn{sort}(a)}) = \ol{\mf{A}}_a,$$ where $\ol{\mf{D}}_a$ are the combinatorial Lascoux polynomials from \cite{monical2016set}, and $\ol{\mf{A}}_a$ are the combinatorial Lascoux atoms from \cite{monical2016set}.
\end{theorem}

The latter two formulas were conjectured by Monical in \cite{monical2016set}, and their equivalence to each other was proven by Monical, Pechenik, and Searles in \cite{monical2021polynomials}. It is unknown whether the polynomials $\ol{\mf{D}}_a$ or $\ol{\mf{A}}_a$ actually have a geometric $K$-theoretic meaning or a representation theoretic meaning, but they are known to have nice combinatorial descriptions (which we again will not repeat here) in terms of ``set-valued" tableaux.

Hivert \cite{hivert2000hecke} and Hicks and Niese \cite{hicks2024quasisymmetric} prove a different analogue of Theorem \ref{thm:classic}, which Hicks and Niese describe as a ``fundamental" analogue:
    
\begin{theorem}[Hivert \cite{hivert2000hecke}; Hicks-Niese \cite{hicks2024quasisymmetric}]\label{thm:hivert-hicks-niese}
    If $a$ is a weak composition of $n$, $\tn{flat}(a)$ is the strong composition formed by removing all 0's from $a$, and $w_a$ is a minimal length permutation that turns $a$ into $\tn{flat}(a)$, $$\til{\pi}_{w_0}(x^a) = F_a,\hspace{1cm}\til{\pi}_{w_a^{-1}}(x^{\tn{flat}(a)}) = \mf{F}_a,\hspace{1cm}\til{\theta}_{w_a^{-1}}(x^{\tn{flat}(a)})=\mf{P}_a,$$ where $F_a$ are the fundamental quasisymmetric functions, $\mf{F}_a$ are the fundamental slides from \cite{assaf2017schubert}, and $\mf{P}_a$ are the fundamental particles from \cite{searles2020polynomial}.
\end{theorem}

The polynomials in Theorem \ref{thm:hivert-hicks-niese} have somewhat simpler combinatorial definitions that do not involve tableaux, which we will given in \S \ref{sec:ops}, and it is known that they have representation theoretic meanings but unknown whether they have geometric meanings. The first two formulas in Theorem \ref{thm:hivert-hicks-niese} were proven by Hivert in \cite{hivert2000hecke}. Hicks and Niese observed that the polynomials arising from the second formula were the fundamental slides defined independently by Searles \cite{searles2020polynomial}, and they proved the third formula above.

We prove the following $K$-analogue of Theorem \ref{thm:hivert-hicks-niese}, which can be thought of as both a fundamental and $K$-theoretic analogue of Theorem \ref{thm:classic}:

\begin{theorem}\label{thm:main}
    Using $a$, $\tn{flat}(a)$, and $w_a$ as above and our $K$-analogues of Hivert's divided difference operators defined in \S\ref{sec:new_ops}, $$\p_{w_0}(x^a) = \ol{F}_a,\hspace{1cm}\p_{w_a^{-1}}(x^{\tn{flat}(a)})=\F_a,\hspace{1cm}\thet_{w_a^{-1}}(x^{\tn{flat}(a)})=\P_a,$$ where $\ol{F}_a$ are the multifundamentals from \cite{lam2007combinatorial}, $\F_a$ are the fundamental glides from \cite{pechenik2019decompositions}, and $\P_a$ are the kaons from \cite{monical2021polynomials}.
\end{theorem}

We define the relevant polynomials in \S\ref{sec:polys} and the relevant divided difference operators in \S\ref{sec:ops}. Then we give examples to illustrate Theorem \ref{thm:main} in \S\ref{sec:ex}, and we give the proof in \S\ref{sec:pf}.

\section{Polynomials}\label{sec:polys}

\subsection{Fundamentals \texorpdfstring{$F_a$}{F} and multifundamentals \texorpdfstring{$\ol{F}_a$}{F}}

A \emph{\tb{\tcb{strong composition}}} is a sequence $a = a_1 a_2\dots a_\ell$ of positive integers. We write $$\tn{set}(a) := \{a_1,a_1+a_2,\dots,a_1+\dots+a_{\ell-1}\}$$ for the set of partial sums of $a$, and $|a|:=a_1+\dots+a_\ell$ for the \emph{\tb{\tcb{size}}} of $a$.

\begin{definition}[Gessel \cite{gessel1984multipartite}]
    The \emph{\tb{\tcb{fundamental quasisymmetric function}}} $F_a(x_1,\dots,x_n)$ is the generating series for all weakly increasing sequences $1\le i_1 \le i_2 \le \dots \le i_{|a|}\le n$ with \emph{strict} increases $i_b < i_{b+1}$ required at the indices $b\in \tn{set}(a)$: $$F_a(x_1,\dots,x_n) := \sum_{\substack{i_1\le \dots \le i_{|a|}, \\ i_b < i_{b+1}\tn{ for }b\in\tn{set}(a)}}x_{i_1}\dots x_{i_{|a|}}.$$
\end{definition}
Note that we are allowing the number of variables $n$ to be chosen independently of $a$.

\begin{definition}[Lam-Pylyavskyy \cite{lam2007combinatorial}]
    The \emph{\tb{\tcb{multifundamental quasisymmetric function}}} $\ol{F}_a$ (introduced as  $\til{L}_a$ in \cite{lam2007combinatorial}) is the generating series for all weakly increasing sequences of \emph{sets} $S_1\le S_2 \le \dots \le S_{|a|}$ with strict increases $S_b < S_{b+1}$ required at indices $b\in \tn{set}(a)$, where $S_i\se\{1,2,\dots,n\}$: $$\ol{F}_a(x_1,\dots,x_n) := \sum_{\substack{S_1\le \dots \le S_{|a|}, \\ S_i\se\{1,2,\dots,n\},\\ S_b<S_{b+1}\tn{ for }b\in\tn{set}(a)}}\beta^{|S_1|+\dots+|S_{|a|}|-|a|}\prod_{i=1}^{|a|}\prod_{j\in S_i} x_j,$$ where the ordering on sets is $$S\le T \iff \max(S) \le \min(T),\hspace{1cm} S < T \iff \max(S) < \min(T).$$ 
\end{definition}

Again, we allow $n$ to be chosen independently of $a$. This definition is typical of combinatorial $K$-analogues, where objects get replaced by corresponding ``set-valued objects," and the $\beta$ variable tracks the ``amount of excess" in the set-valued objects. Setting $\beta = 0$, we recover $\ol{F}_a(x_1,\dots,x_n;0) = F_a(x_1,\dots,x_n)$, which is also a property we generally want for $K$-analogues.

\begin{example}
    For $a = 121$ and $n = 4$, we have $\tn{set}(a)=\{1,3\},$ so we need weakly increasing sequences of 4 subsets of $\{1,2,3,4\}$ with strict increases from index 1 to 2 and 3 to 4. The possible sequences are $$\begin{array}{cccccc}
        1|2|2|3 & 1|2|2|4 & 1|2|3|4 & 1|3|3|4 &1|2|23|4 & 1|23|3|4,
    \end{array}$$
    where the vertical bars represent boundaries between sets. It follows that
    \begin{align*}
        F_{121}(x_1,x_2,x_3,x_4) &= x_1 x_2^2 x_3 + x_1 x_2^2 x_4 + x_1x_2 x_3 x_4 + x_1 x_3^2 x_4, \\
        \ol{F}_{121}(x_1,x_2,x_3,x_4) &= x_1 x_2^2 x_3 + x_1 x_2^2 x_4 + x_1x_2 x_3 x_4 + x_1 x_3^2 x_4 + \beta(x_1 x_2^2 x_3 x_4 + x_1 x_2 x_3^2 x_4).
    \end{align*}
\end{example}

\subsection{Fundamental slides \texorpdfstring{$\mf{F}_a$}{F} and fundamental glides \texorpdfstring{$\F_a$}{F}}

A \emph{\tb{\tcb{weak composition}}} is a finite sequence $a = a_1 a_2\dots a_n$ of \emph{nonnegative} integers. The \emph{\tb{\tcb{flat}}} $\tn{flat}(a)$ is the strong composition formed by deleting all the 0's from $a$. A composition $b$ \emph{\tb{\tcb{refines}}} $a$ if $b$ is formed by splitting some of the parts of $a$ into multiple parts, $b$ \emph{\tb{\tcb{dominates}}} $a$ if $b_1+\dots+b_i\ge a_1+\dots+a_i$ for all $i$.

\begin{definition}[Assaf-Searles \cite{assaf2017schubert}]
    We say $b$ is a \emph{\tb{\tcb{slide}}} of $a$ if $b$ dominates $a$ and $\tn{flat}(b)$ refines $\tn{flat}(a)$. The \emph{\tb{\tcb{fundamental slide polynomial}}} $\mf{F}_a$ is the generating series for slides of $a$: $$\mf{F}_a(x_1,\dots,x_n):= \sum_{\substack{b\tn{ dominates }a, \\ \tn{flat}(b)\tn{ refines }\tn{flat}(a)}}x^b,$$ where for a composition $b = b_1\dots b_\ell$ we write $x^b := x_1^{b_1}\dots x_\ell^{b_\ell}$.

\end{definition}

A \emph{\tb{\tcb{weak komposition}}} is a weak composition where the nonzero parts are colored either black or red. The \emph{\tb{\tcb{excess}}} $\tn{ex}(b)$ of $b$ is the number of red entries.

\begin{definition}[Pechenik-Searles \cite{pechenik2019decompositions}]
For a weak composition $a$ with nonzero entries at indices $i_1<\dots<i_\ell$, a \emph{\tb{\tcb{glide}}} of $a$ is a weak komposition $b$ that can be partitioned into $\ell$ blocks, such that the $j$th block:
\begin{enumerate}
    \item Has a black number for its first nonzero entry.
    \item Forms a weak komposition with (sum of entries) $-$ excess $ = a_{n_j}$.
    \item Has its rightmost entry weakly to the left of position $n_j$.
\end{enumerate}

\begin{definition}[Pechenik-Searles \cite{pechenik2019decompositions}]
    The \emph{\tb{\tcb{fundamental glide polynomial}}} $\F_a$ is the generating series for glides of $a$:
$$\F_a(x_1,\dots,x_n) := \sum_{b\tn{ a glide of }a}\beta^{\tn{ex}(b)}x^b.$$ 
\end{definition}
\end{definition}

The slides of $a$ are just the glides with no red parts, so $\F_a$ specializes to $\mf{F}_a$ when $\beta = 0$.

\begin{example}\label{ex:slide-glide}
    For $a = 021$, the glides of $a$ are listed below, where the vertical bars indicate where the blocks end, and the red entries are also shown bolded with a bar over them:
    $$\begin{array}{ccccccc}
        2|10 & 2|01 & 11|1 & 02|1 & 2|1\boldsymbol{\color{red}\ol{1}} & 2\boldsymbol{\color{red}\ol{1}}|1 & 1\boldsymbol{\color{red}\ol{2}}|1.
    \end{array}$$
    Thus, the slide and glide polynomials are, respectively,
    \begin{align*}
        \mf{F}_{021} &= x_1^2 x_2 + x_1^2 x_3 + x_1 x_2 x_3 + x_2^2 x_3, \\
        \F_{021} &= x_1^2 x_2 + x_1^2 x_3 + x_1 x_2 x_3 + x_2^2 x_3 + \beta(2x_1^2 x_2 x_3 + x_1 x_2^2 x_3).
    \end{align*}
\end{example}

\subsection{Fundamental particles \texorpdfstring{$\mf{P}_a$}{P} and kaons \texorpdfstring{$\P_a$}{P}}

\begin{definition}[Monical-Pechenik-Searles \cite{searles2020polynomial,monical2021polynomials}]
    A glide of $a$ is \emph{\tb{\tcb{mesonic}}} (or a slide is called \emph{\tb{\tcb{fixed}}}, in the terminology of \cite{searles2020polynomial}) if for all $j$, the $j$th block ends exactly at position $n_j$ and has a nonzero number at that position. 
\end{definition}

\begin{definition}[Monical-Pechenik-Searles \cite{searles2020polynomial,monical2021polynomials}]
    The \emph{\tb{\tcb{fundamental particles}}} $\mf{P}_a$ (introduced as $\mf{L}_a$ in \cite{searles2020polynomial}) are the generating series for fixed slides: $$\mf{P}_a(x_1,\dots,x_n) := \sum_{b\tn{ a fixed slide of }a}x^b,$$ and the \emph{\tb{\tcb{kaons}}} are the generating series for mesonic glides: $$\P_a(x_1,\dots,x_n) := \sum_{b\tn{ a mesonic glide of }a}\beta^{\tn{ex}(b)}x^b.$$ 
\end{definition}

Searles shows in \cite{searles2020polynomial} that the slides are always positive sums of particles, and Monical, Pechenik, and Searles show in \cite{monical2021polynomials} that the glides are always positive sums of kaons.

\begin{example}
    With $a = 021$ as in Example \ref{ex:slide-glide}, the mesonic glides are:
    $$\begin{array}{ccccccc}
        11|1 & 02|1 & 2\boldsymbol{\color{red}\ol{1}}|1 & 1\boldsymbol{\color{red}\ol{2}}|1,
    \end{array}$$
    so the particle and kaon are, respectively,
    \begin{align*}
        \mf{P}_{021} &= x_1x_2x_3 + x_2^2x_3, \\
        \P_{021} &= x_1x_2x_3 + x_2^2x_3 + \beta(x_1^2 x_2 x_3 + x_1 x_2^2 x_3).
    \end{align*}
\end{example}

\section{Divided difference operators}\label{sec:ops}

\subsection{Classic divided difference operators}\label{sec:classic_ops}

Let $s_i$ be the involution (and ring isomorphism) on polynomials that swaps the variables $x_i$ and $x_{i+1}$, so $$s_i(x_1^{a_1}\dots x_i^{a_i} x_{i+1}^{a_{i+1}}\dots x_n^{a_n}) := x_1^{a_1}\dots x_i^{a_{i+1}} x_{i+1}^{a_i}\dots x_n^{a_n}.$$ The classic divided difference operators are:
\begin{definition}
    \begin{align*}
    \partial_i &:= \frac{1-s_i}{x_i-x_{i+1}}, \\
    \pi_i &:= \partial_i x_i = \frac{x_i-x_{i+1}s_i}{x_i-x_{i+1}}, \\
    \theta_i &:= \pi_i - 1 = x_{i+1}\partial_i. 
\end{align*}
\end{definition}
These operators satisfy the same braid relations and commutation relations as the adjacent transpositions $\sigma_i = (i \ i+1)$ that generate the symmetric group $\mathfrak{S}_n$, namely,
$$\begin{array}{ccc}
    \partial_i\partial_{i+1}\partial_i = \partial_{i+1}\partial_i\partial_{i+1},
     \ \ & \pi_i \pi_{i+1}\pi_i = \pi_{i+1}\pi_i\pi_{i+1}, \ \ & \theta_i\theta_{i+1}\theta_i = \theta_{i+1}\theta_i\theta_{i+1},
\end{array}$$ and whenever $|i-j|\ge 2$,
$$\begin{array}{ccc}
    \partial_i\partial_j = \partial_j\partial_i, \ \ \ & \pi_i \pi_j = \pi_j\pi_i, \ \ \ & \theta_i\theta_j = \theta_j\theta_i.
\end{array}$$ However, unlike the adjacent transpositions, the squaring relations for these operators are
$$\begin{array}{ccc}
    \partial_i^2 = 0, \ \ \ & \pi_i^2 = \pi_i, \ \ \ & \theta_i^2 = -\theta_i.
\end{array}$$ Thus, $\partial_i$ behaves like a boundary operator and $\pi_i$ like a projection, hence the notation $\partial$ and $\pi$. Lascoux and Sch\"utzenberger call $\pi$ the \emph{\tb{\tcb{convex symmetrizer}}}. This name makes sense because if we assume without loss of generality that $a>b$, then $$\pi_i(x_i^a x_{i+1}^{b}) = \frac{x_i^{a+1}x_{i+1}^b - x_i^bx_{i+1}^{a+1}}{x_i-x_{i+1}} = x_i^{a} x_{i+1}^{b} + x_i^{a-1}x_{i+1}^{b+1}+\dots+x_i^{b}x_{i+1}^{a}$$ is a symmetric polynomial in $a$ and $b$ consisting of all monomials $x_i^c x_{i+1}^d$ where $(c,d)$ is a lattice point on the line segment connecting $(a,b)$ to $(b,a)$. Thus, $\pi_i$ projects the full set polynomials onto the set of polynomials that are symmetric in $x_i$ and $x_{i+1}$, in a way analogous to taking the convex hull of $x_i^a x_{i+1}^b$ and $x_i^b x_{i+1}^a$.

More generally, given any permutation $w\in\mathfrak{S}_n$ and a reduced word $w = \sigma_{i_1}\dots \sigma_{i_\ell}$ for $w$, the operators $\partial_w$, $\pi_w$, and $\theta_w$ are given by
$$\begin{array}{ccc}
   \partial_w := \partial_{i_1}\dots\partial_{i_\ell},  \ \ & \pi_w := \pi_{i_1}\dots\pi_{i_\ell}, \ \ & \theta_w := \theta_{i_1}\dots\theta_{i_\ell}.
\end{array}$$ The fact that these are independent of the choice of reduced word comes from the braid and commutation relations above.

\subsection{Hivert's divided difference operators}\label{sec:Hivert_ops}

Hivert's operator $\til{s}_i$ is the linear operator on polynomials that acts on the monomial $x^a = x_1^{a_1}\dots x_n^{a_n}$ by
$$\til{s}_i(x^a):= \begin{cases}
    s_i(x^a) &\tn{if }a_i=0\text{ or }a_{i+1}=0,\\
    x^a &\tn{else.}
\end{cases}$$
That is, $\til{s}_i$ swaps $x_i$ and $x_{i+1}$ \emph{only if $x_i$ and $x_{i+1}$ do not both appear in the monomial}, and otherwise it keeps the monomial the same. Unlike $s_i$, note that $\til{s}_i$ is not a ring homomorphism, since if $a\ne b$, $$\til{s}_i(x_i^a)\til{s}_i(x_{i+1}^b) = x_{i+1}^a x_i^b \ne x_i^a x_{i+1}^b = \til{s}_i(x_i^a x_{i+1}^b).$$ Hivert defines the following analogues of $\partial_i$, $\pi_i,$ and $\theta_i$:
\begin{definition}[Hivert \cite{hivert2000hecke}]
\begin{align*}
    \til{\partial}_i &:= \frac{1-\til{s}_i}{x_i-x_{i+1}}, \\
    \til{\pi}_i &:= \frac{x_i-x_{i+1}\til{s}_i}{x_i-x_{i+1}}, \\
    \til{\theta}_i &:= \til{\pi}_i-1 = x_{i+1}\til{\partial}_i.
\end{align*}
\end{definition}
\noindent Hivert actually writes $\boldsymbol{\sigma}$ for $\til{s}$, $\boldsymbol{\pi}$ for $\til{\pi}$, and $\overline{\boldsymbol{\pi}}$ for $\til{\theta}$, but we follow the notation of \cite{hicks2024quasisymmetric} instead. 

Hivert's operators $\til{\pi}$ and $\til{\theta}$ satisfy the same braid, commutation, and squaring relations as $\pi$ and $\theta$. Thus, for a general permutation $w\in \mathfrak{S}_n$, we can define $\til{\pi}_w$ and $\til{\theta}_w$ in the same way as for $\pi_w$ and $\theta_w$,  namely, $$\til{\pi}_w := \til{\pi}_{i_1}\dots\til{\pi}_{i_\ell}, \hspace{1cm}\til{\theta}_w := \til{\theta}_{i_1}\dots\til{\theta}_{i_\ell}$$ and the result will be independent of the choice of reduced word $w = \sigma_{i_1}\dots \sigma_{i_\ell}.$

Note that Hivert's definition of $\til{\pi}_i$ is \emph{not} equivalent to $\til{\partial}_i x_i$, since if $a,b\ne 0$ we have $$\til{\pi}_i(x_i^a x_{i+1}^b) = \frac{x_i\cdot x_i^a x_{i+1}^b - x_{i+1}\cdot \til{s}_i(x_i^a x_{i+1}^b)}{x_i-x_{i+1}} = \frac{x_i^{a+1}x_{i+1}^b - x_i^a x_{i+1}^{b+1}}{x_i-x_{i+1}} = x_i^a x_{i+1}^b,$$ while $$\til{\partial}_ix_i(x_i^ax_{i+1}^b) = \frac{x_i^{a+1}x_{i+1}^b - \til{s}_i(x_i^{a+1}x_{i+1}^b)}{x_i-x_{i+1}} = 0.$$

\subsection{New \texorpdfstring{$K$}{K}-analogues of Hivert's divided difference operators}\label{sec:new_ops}

We define the following $K$-analogues of Hivert's divided difference operators:
\begin{definition}
    \begin{align*}
        \ol{\til{\pi}}_i &:= \frac{x_i - x_{i+1}\til{s}_i + \beta x_ix_{i+1}(1-\til{s}_i)}{x_i-x_{i+1}}, \\
        \ol{\til{\theta}}_i &:= \ol{\til{\pi}}_i - 1 = x_{i+1}(1+\beta x_i)\til{\partial}_i.
    \end{align*}
\end{definition}
\noindent Note that these specialize to $\til{\pi}_i$ and $\til{\theta}_i$ at $\beta = 0$, as we would want. Note also that if $a,b\ne 0$ and $f$ is a polynomial not involving $x_i$ or $x_{i+1}$, $$\p_i(x_i^a x_{i+1}^b f) = x_i^a x_{i+1}^bf,\hspace{1cm}\thet_i(x_i^a x_{i+1}^bf) = 0.$$ That is, $\p_i$ fixes monomials involving both $x_i$ and $x_{i+1}$ and $\thet_i$ sends such monomials to 0.

For $w\in\mf{S}_n$, we would like to define $\ol{\til{\pi}}_w$ and $\ol{\til{\theta}}_w$ in the same way as above, namely,
$$\ol{\til{\pi}}_w := \ol{\til{\pi}}_{i_1}\dots\ol{\til{\pi}}_{i_\ell}, \hspace{1cm}\ol{\til{\theta}}_w := \ol{\til{\theta}}_{i_1}\dots\ol{\til{\theta}}_{i_\ell},$$ where $w = \sigma_{i_1}\dots\sigma_{i_\ell}$ is a reduced word for $w$. However, for these definitions to be independent of the choice of reduced word, we need to check that the $\ol{\til{\pi}}_i$'s and $\ol{\til{\theta}}_i$'s satisfy the braid and commutation relations:
\begin{claim}
The operators defined above satisfy
$$\p_i^2 = \p_i, \hspace{1cm} \thet_i^2 = -\thet_i,\hspace{1cm}\ol{\til{\pi}}_i \ol{\til{\pi}}_{i+1}\ol{\til{\pi}}_i = \ol{\til{\pi}}_{i+1}\ol{\til{\pi}}_i\ol{\til{\pi}}_{i+1}, \hspace{1cm} \ol{\til{\theta}}_i \ol{\til{\theta}}_{i+1}\ol{\til{\theta}}_i = \ol{\til{\theta}}_{i+1}\ol{\til{\theta}}_i\ol{\til{\theta}}_{i+1}=0,$$ and for $|i-j|\ge 2$ they satisfy $\ol{\til{\pi}}_i\ol{\til{\pi}}_j = \ol{\til{\pi}}_j\ol{\til{\pi}}_i$ and $\ol{\til{\theta}}_i \ol{\til{\theta}}_j = \ol{\til{\theta}}_j \ol{\til{\theta}}_i$.
\end{claim}

\begin{proof}
We check the commutation, then squaring, then braid relations:
\begin{itemize}
    \item \tb{Commutation:} The commutation relations hold since if $f$ is a polynomial not involving $x_i$ or $x_{i+1}$, we have $\ol{\til{\pi}}_i(x_i^a x_{i+1}^b f) = \ol{\til{\pi}}_i(x_i^a x_{i+1}^b)f$, so if $|i-j|\ge 2$ and $f$ does not involve $x_i,$ $x_{i+1}$, $x_j,$ or $x_{j+1}$, then $$\ol{\til{\pi}}_i\ol{\til{\pi}}_j(x_i^a x_{i+1}^b x_j^c x_{j+1}^d f) = \ol{\til{\pi}}_i(x_i^a x_{i+1}^b\ol{\til{\pi}}_j(x_j^c x_{j+1}^d)f) = \ol{\til{\pi}}_i(x_i^a x_{i+1}^b)\ol{\til{\pi}}_j(x_j^c x_{j+1}^d)f,$$ and similarly, $$\ol{\til{\pi}}_j\ol{\til{\pi}}_i(x_i^a x_{i+1}^b x_j^c x_{j+1}^d f) = \ol{\til{\pi}}_j(\ol{\til{\pi}}_i(x_i^a x_{i+1}^b)x_j^c x_{j+1}^df) = \ol{\til{\pi}}_i(x_i^a x_{i+1}^b)\ol{\til{\pi}}_j(x_j^c x_{j+1}^d)f.$$ Thus, $\ol{\til{\pi}}_i\ol{\til{\pi}}_j = \ol{\til{\pi}}_j\ol{\til{\pi}}_i$, and the same argument shows that $\ol{\til{\theta}}_i \ol{\til{\theta}}_j = \ol{\til{\theta}}_j \ol{\til{\theta}}_i$.
    \item \tb{Squaring:} For the squaring relations, consider $\thet_i$ first. For monomials involving both $x_i$ and $x_{i+1}$, we have $\thet_i^2 = -\thet_i = 0$. Since applying $\thet_i$ commutes with multiplication by any polynomial not involving $x_i$ or $x_{i+1}$, it suffices to consider monomials of the form $x_i^a$ or $x_{i+1}^a$. Note that \begin{align*}
        \thet_i(x_i^a) &= \frac{x_{i+1}(1+\beta x_i)(x_i^a - x_{i+1}^a)}{x_i - x_{i+1}} \\
        &= x_i^{a-1}x_{i+1} + \dots + x_i x_{i+1}^{a-1} + x_{i+1}^a + \beta(x_i^a x_{i+1} + \dots + x_i x_{i+1}^a), \\
        \thet_i(x_{i+1}^a) &= \frac{x_{i+1}(1+\beta x_i)(x_{i+1}^a - x_i^a)}{x_i - x_{i+1}} = -\thet_i(x_i^a).
    \end{align*}
    When we apply $\thet_i$ again to $\thet_i(x_i^a)$, all terms will become 0 except the $x_{i+1}^a$ term, since all other terms involve both an $x_i$ and an $x_{i+1}$, and similarly for when we apply $\thet_i$ again to $\thet_i(x_{i+1}^a)$. Thus, $$\thet_i^2(x_i^a) = \thet_i(x_{i+1}^a) = -\thet_i(x_i^a),\hspace{1cm}\thet_i^2(x_{i+1}^a) = \thet_i^2(-x_{i+1}^a) = -\thet_i^2(x_{i+1}^a).$$ Then for $\p$, using $\p_i = \thet_i+1$, we get $$\p_i^2 = (\thet_i+1)^2 = \thet_i^2 + 2\thet_i + 1 = -\thet_i + 2\thet_i + 1 = \thet_i+1=\p_i,$$ as claimed.
    
    \item \tb{Braid relations:} For the braid relations, we again consider $\thet$ first. Using $\thet_i = x_{i+1}(1+\beta x_i)\del_i$ and the fact that multiplication by variables besides $x_i$ or $x_{i+1}$ commutes with $\del_i$, we get 
    \begin{align*}
        \thet_{i+1}\thet_i\thet_{i+1} &= x_{i+2}(1+\beta x_{i+1})\del_{i+1} x_{i+1}(1+\beta x_{i})\del_{i} x_{i+2}(1+\beta x_{i+1})\del_{i+1} \\
        &= x_{i+2}(1+\beta x_{i+1})\del_{i+1} x_{i+1}x_{i+2}(1+\beta x_{i})\del_{i} (1+\beta x_{i+1})\del_{i+1} = 0,
    \end{align*}
    since the $\del_{i+1} x_{i+1}x_{i+2}$ part is 0. On the other hand,
    \begin{align*}
        \thet_i \thet_{i+1}\thet_i &= x_{i+1}(1+\beta x_i)\del_{i} x_{i+2}(1+\beta x_{i+1})\del_{i+1} x_{i+1}(1+\beta x_{i})\del_{i}.
    \end{align*}
    For a monomial involving both $x_i$ and $x_{i+1}$, the $\del_i$ on the right automatically sends it to 0. For a monomial involving $x_{i+2}$, the $x_{i+2}$ is unaffected by the rightmost $\del_i$, so the $\del_{i+1}$ sends it to 0. Thus, it remains to consider monomials that do not involve $x_{i+2}$ but involve exactly one of $x_i$ or $x_{i+1}$. Since the rightmost operator is $\del_i$ and $\del_i(x_{i+1}^a) = -\del_i(x_i^a)$, we will have $\thet_i\thet_{i+1}\thet_i(x_{i+1}^a) = -\thet_i\thet_{i+1}\thet_i(x_{i}^a),$ so it suffices to consider $\thet_i\thet_{i+1}\thet_i(x_{i}^a)$. First moving some commuting factors to the left, we get
    \begin{align*}
        \thet_i \thet_{i+1}\thet_i(x_i^a) &= x_{i+1}x_{i+2}(1+\beta x_i)\del_i(1+\beta x_{i+1})(1+\beta x_i)\del_{i+1} x_{i+1}\del_i(x_i^a) \\
        &= x_{i+1}x_{i+2}(1+\beta x_i)\del_i(1+\beta x_{i+1})(1+\beta x_i)\del_{i+1} x_{i+1}\frac{x_i^{a}-x_{i+1}^a}{x_i-x_{i+1}} \\
        &= x_{i+1}x_{i+2}(1+\beta x_i)\del_i(1+\beta x_{i+1})(1+\beta x_i)\del_{i+1} (x_i^{a-1}x_{i+1}+x_i^{a-2}x_{i+1}^2+\dots+x_{i+1}^{a}).
    \end{align*}
    Applying $\del_{i+1}$ to that rightmost sum gives 
    \begin{align*}
        x_i^{a-1} + x_i^{a-2}(x_{i+1}+x_{i+2})+\dots+(x_{i+1}^{a-1}+x_{i+1}^{a-1}x_{i+2}+\dots+x_{i+2}^{a-1}).
    \end{align*}
    The next operator is $\del_i$, which sends all terms involving both $x_i$ and $x_{i+1}$ to 0, so the terms that do not get zeroed out are $$(1+\beta x_i)(x_i^{a-1}+x_i^{a-2}x_{i+2}+\dots+x_{i+2}^{a-1}) + (1+\beta x_{i+1})(x_{i+1}^{a-1}+x_{i+1}^{a-2}x_{i+2}+\dots+x_{i+2}^{a-1}),$$ which will also get sent to 0 by $\del_i$ because it is symmetric in $x_i$ and $x_{i+1}$, and $\del_i$ sends all polynomials that are symmetric in $x_i$ and $x_{i+1}$ to 0. This covers all cases, so it follows that $$\thet_i\thet_{i+1}\thet_i = \thet_{i+1}\thet_i\thet_{i+1}=0.$$

    Finally, for $\p$, we again use the fact that $\p_i = \thet_i+1$ together with the relation $\thet_i\thet_{i+1}\thet_i = \thet_{i+1}\thet_i\thet_{i+1}=0$ above and the relations $\thet_i^2=-\thet_i$ and $\thet_{i+1}^2 = -\thet_{i+1}$ to get
    \begin{align*}
        \p_i \p_{i+1}\p_i &= (\thet_i + 1)(\thet_{i+1}+1)(\thet_i + 1) \\
        &= \thet_i \thet_{i+1}\thet_i + \thet_i\thet_{i+1}+\thet_{i+1}\thet_i+\thet_i^2 + \thet_{i+1}+2\thet_i+1 \\
        &= 0 + \thet_i\thet_{i+1}+\thet_{i+1}\thet_i - \thet_i + \thet_{i+1}+2\thet_i+1 \\
        &= 0 + \thet_i\thet_{i+1}+\thet_{i+1}\thet_i + \thet_{i+1}+\thet_i + 1 \\
        &= \thet_{i+1}\thet_i\thet_{i+1} + \thet_i\thet_{i+1}+\thet_{i+1}\thet_i + \thet_{i+1}^2 + 2\thet_{i+1}+\thet_i + 1 \\
        &= (\thet_{i+1}+1)(\thet_i+1)(\thet_{i+1}+1) \\
        &= \p_{i+1}\p_i\p_{i+1},
    \end{align*}
    as desired.
\end{itemize}

\end{proof}

\section{Examples of Theorem \ref{thm:main}}\label{sec:ex}

\subsection{Example of \texorpdfstring{$\ol{\til{\pi}}_{w_0}(x^a) = \ol{F}_a$}{first formula}}

Let $a = 210$ and $n=3$. Then $\tn{set}(a) = \{2\},$ and the possible weakly increasing sequences of subsets of $\{1,2,3\}$ with a strict increase at index 2 are 
$$\begin{array}{ccccccc}
    1|1|2 & 1|1|3 & 1|2|3 & 2|2|3 & 1|1|23 & 1|12|3 & 12|2|3
\end{array},$$ so the multifundamental polynomial is $$\ol{F}_a(x_1,x_2,x_3) = x_1^2 x_2 + x_1^2 x_3 + x_1 x_2 x_3 + x_2^2 x_3 + \beta(2 x_1^2 x_2 x_3 + x_1 x_2^2 x_3).$$ Since $n=3$, the longest permutation is $w_0 = \sigma_1\sigma_2\sigma_1\in \mf{S}_3$. Thus we get 
\begin{align*}
    \p_{w_0}(x^a) &= \p_1\p_2\p_1(x_1^2 x_2) = \p_1\p_2(x_1^2 x_2) \\
    &= \p_1\frac{x_1^2x_2^2 - x_1^2x_3^2 + \beta x_2x_3(x_1^2 x_2 - x_1^2 x_3)}{x_2 - x_3} \\
    &= \p_1(x_1^2 x_2 + x_1^2 x_3 + \beta x_1^2 x_2x_3) \\
    &= x_1^2 x_2 + \beta x_1^2 x_2 x_3 + \p_1(x_1^2 x_3) \\
    &= x_1^2 x_2 + \beta x_1^2 x_2 x_3 + \frac{x_1^3 x_3 - x_2^3 x_3 + \beta x_1 x_2(x_1^2 x_3 - x_2^2 x_3)}{x_1 - x_2} \\
    &= x_1^2 x_2 + \beta x_1^2 x_2 x_3 + x_1^2 x_3 + x_1x_2x_3 + x_2^2 x_3 + \beta(x_1^2 x_2 x_3 + x_1 x_2^2 x_3) \\
    &= \ol{F}_a(x_1,x_2,x_3),
\end{align*}
making use of the fact that $\p_i$ fixes all monomials involving both $x_i$ and $x_{i+1}$.

\subsection{Example of \texorpdfstring{$\ol{\til{\pi}}_{w_a^{-1}}(x^{\tn{flat}(a)}) = \ol{\mf{F}}_a$}{second formula}}

Let $a = 012$. Then the possible glides of $a$ are
$$\begin{array}{ccccccc}
    1|20 & 1|11 & 1|02 & 01|2 & 1|\boldsymbol{\color{red}{\ol{2}}}1 & 1\boldsymbol{\color{red}{\ol{1}}}|2 & 1|1\boldsymbol{\color{red}{\ol{2}}},
\end{array}$$
so the fundamental glide polynomial is $$\F_a = x_1 x_2^2 + x_1x_2x_3 + x_1x_3^2 + x_2x_3^2 + \beta(x_1x_2^2x_3 + 2x_1x_2x_3^2).$$ We have $\tn{flat}(a) = 120$, so $$\sigma_2\sigma_1(a) = \sigma_2\sigma_1(012) = \sigma_2(102) = 120 = \tn{flat}(a),$$ so $w_a = \sigma_2\sigma_1$ and thus $w_a^{-1} = \sigma_1\sigma_2$. Then we can check that \begin{align*}
    \p_{w_a^{-1}}(x^{\tn{flat}(a)}) &= \p_1\p_2(x_1x_2^2) \\
    &= \p_1\frac{x_1 x_2^3 - x_1 x_3^3 + \beta x_2x_3(x_1 x_2^2 - x_1 x_3^2)}{x_2-x_3} \\
    &= \p_1(x_1 x_2^2 + x_1x_2x_3 + x_1x_3^2 + \beta (x_1 x_2^2 x_3 + x_1x_2 x_3^2)) \\
    &= x_1 x_2^2 + x_1x_2x_3 + \beta (x_1 x_2^2 x_3 + x_1x_2 x_3^2) + \p_1(x_1 x_3^2) \\
    &= x_1 x_2^2 + x_1x_2x_3 + \beta (x_1 x_2^2 x_3 + x_1x_2 x_3^2) + \frac{x_1^2 x_3^2 - x_2^2 x_3^2 + \beta x_1 x_2(x_1x_3^2 - x_2x_3^2)}{x_1 - x_2} \\
    &= x_1 x_2^2 + x_1x_2x_3 + \beta (x_1 x_2^2 x_3 + x_1x_2 x_3^2) + x_1 x_3^2 + x_2 x_3^2 + \beta x_1 x_2 x_3^2 = \F_a,
\end{align*}
where we again used the fact that $\p_i$ fixes monomials involving both $x_i$ and $x_{i+1}$.

\subsection{Example of \texorpdfstring{$\ol{\til{\theta}}_{w_a^{-1}}(x^{\tn{flat}(a)}) = \ol{\mf{P}}_a$}{third formula}}

Let $a = 012$ as above, so $\tn{flat}(a) = 120$ and $w_a^{-1} = \sigma_1\sigma_2$. The only mesonic glides of $a$ are $01|2$ and $1\boldsymbol{\color{red}\ol{1}}|2$, so the kaon is $$\P_a = x_2 x_3^2 + \beta x_1 x_2 x_3^2.$$ Then we get
\begin{align*}
    \thet_{w_a^{-1}}(x^{\tn{flat}(a)}) &= \thet_1\thet_2(x_1 x_2^2) \\
    &= \thet_1 \frac{x_3(1+\beta x_2)(x_1 x_2^2 - x_1 x_3^2)}{x_2 - x_3} \\
    &= \thet_1(x_3(1+\beta x_2)(x_1 x_2 + x_1x_3)) = \thet_1(x_1 x_3^2) \\
    &= \frac{x_2(1+\beta x_1)(x_1 x_3^2 - x_2 x_3^2)}{x_1 - x_2} \\
    &= x_2 x_3^2 + \beta x_1 x_2 x_3^2 = \P_a,
\end{align*}
using the fact that $\thet_i$ sends all monomials involving both $x_i$ and $x_{i+1}$ to 0.

\section{Proof of Theorem \ref{thm:main}}\label{sec:pf}

\subsection{Proof that \texorpdfstring{$\ol{\til{\pi}}_{w_0}(x^a) = \ol{F}_a$}{first formula}}\label{sec:proof1}

We will show in general that $$\p_{w_0(n)}(x^a) = \ol{F}_a(x_1,\dots,x_n)$$ where the number of variables $n$ may be chosen independently of $a$, and $w_0(n)$ is the longest permutation in $\mf{S}_n$. We use induction on the number $n$ of variables $x_1,x_2,\dots,x_n$. If $a = a_1\dots a_\ell$ has length $\ell$, then $\ol{F}_a(x_1,\dots,x_n) = 0$ for $n<\ell$, since it is impossible to have $\ell$ strict increases if we are using subsets of $\{1,2,\dots,n\}$ with $n<\ell$. Thus, we may take $n=\ell$ as the base case. In that case, the only way to have $\ell$ strict increases is if all the sets have size 1, so $\ol{F}_a(x_1,\dots,x_\ell) = x^a$. Thus, we need to show that $\p_{w_0(\ell)}(x^a) = x^a$ for $n=\ell$. 

Since $w_0(\ell)$ is the longest permutation in $\mf{S}_\ell$, any reduced word for $w_0(\ell)$ is a sequence of transpositions $\sigma_i$ with $1\le i\le \ell-1$, so it suffices to show that every $\p_i$ with $1\le i\le \ell-1$ fixes $x^a$. Since $a$ is a strong composition, it has no internal 0's, so in particular $a_i\ne0$ and $a_{i+1}\ne 0$. But $\p_i$ fixes all monomials involving both $x_i$ and $x_{i+1}$, because $\til{s}_i$ acts by the identity on all such monomials, so the action of $\p_i$ on such a monomial is 
\begin{equation}\label{eqn:pi_fixed_things}
    \p_i = \frac{x_i - x_{i+1}\til{s}_i + \beta x_i x_{i+1}(1-\til{s}_i)}{x_i - x_{i+1}} = \frac{x_i-x_{i+1} + \beta x_i x_{i+1}(1-1)}{x_i-x_{i+1}} = 1.
\end{equation}
Thus, all the $\p_i$'s making up $\p_{w_0(\ell)}$ fix $x^a$, so $\p_{w_0(\ell)}$ fixes it as well, completing the base case.

For the inductive step, note that $w_0(n) = \sigma_1\sigma_2\dots\sigma_{n-1}w_0(n-1)$. We may assume from the inductive hypothesis that $\p_{w_0(n-1)}(x^a) = \ol{F}_a(x_1,\dots,x_{n-1})$, so it suffices to show that \begin{equation}\label{eqn:F}
    \ol{F}_a(x_1,\dots,x_n) = \p_1\p_2\dots\p_{n-1}(\ol{F}_a(x_1,\dots,x_{n-1})).
\end{equation}
On the right side, we are starting with the generating series $\ol{F}_a(x_1,\dots,x_{n-1})$ for weakly increasing sequences of subsets $S_1\le \dots \le S_{|a|}$ with strict increases at the indices in $\tn{set}(a)$ and $S_i\se \{1,2,\dots,n-1\}$. We want to show that after applying $\p_1\dots \p_{n-1}$ we gain all terms where the sets $S_i$ may also contain $n$.

\begin{lemma}\label{lem}
    For each $0\le k\le \ell,$ after applying the operators $\p_{n-k}\dots \p_{n-1}$ (which is an empty sequence if $k=0$) we get the generating series for all nondecreasing set sequences counted on the left side of (\ref{eqn:F}) that also satisfy the additional restriction $\max(S_{a_1+\dots+a_{\ell-k}})\le n-k-1.$
\end{lemma}

\begin{proof}
    We use induction on $k$. For the base case $k=0$, the statement is that $\max(S_{a_1+\dots+a_{\ell}})=\max(S_{|a|})\le n-1$, which is precisely what we are assuming by saying that to start only the variables $x_1,\dots,x_{n-1}$ may be used. Thus, the base case holds.

    For the inductive step, assume the statement holds before applying $\p_{n-k}$ for some $1\le k\le \ell$, so we need to show that it still holds after applying $\p_{n-k}.$ The new terms we are claiming are included after applying $\p_{n-k}$ but not before applying it are the ones that still satisfy the restriction $\max(S_{a_1+\dots+a_{\ell-k}}) \le n-k-1$ but no longer satisfy the restriction $\max(S_{a_1+\dots+a_{\ell-k+1}}) \le n-k$, which means we have $\max(S_{a_1+\dots+a_{\ell-k}}) \le n-k-1$ but $\max(S_{a_1+\dots+a_{\ell-k+1}}) \ge n-k+1$. Note that there are $k-1$ indices in $\tn{set}(a)$ that are greater than or equal to $a_1+\dots + a_{\ell-k+1}$, which means there are $k-1$ strict increases required after $S_{a_1+\dots+a_{\ell-k+1}}$, so as long as all elements used are at most $n$, we must always have $\max(S_{a_1+\dots+a_{\ell-k+1}})\le n-k+1$. The only way it can exactly equal $n-k+1$ is if all sets after $S_{a_1+\dots+a_{\ell-k+1}}$ have size 1 and are as large as possible given the strict increase requirements, so that the final portion of the sequence of sets looks something like $$\dots|\underbrace{\dots\le n-k-1}_{a_{\ell-k}\tn{ sets}}|\underbrace{\dots n-k+1}_{a_{\ell-k+1}\tn{ sets}}|\underbrace{n-k+2|\dots|n-k+2}_{a_{\ell-k+2}\tn{ sets}}|\dots|\underbrace{n-1|\dots|n-1}_{a_{\ell-1}\tn{ sets}}|\underbrace{n|\dots|n}_{a_\ell\tn{ sets}},$$ where the vertical bars represent boundaries between the subsets $S_i$. (Note that the set $S_{a_1+\dots+a_{\ell-k+1}}$ may have size greater than 1 if it contains additional smaller elements besides $n-k+1$.)
    
    For any such term that we hope to get from applying $\p_{n-k}$, there is a unique corresponding term that exists before applying $\p_{n-k}$ obtained by replacing all the $(n-k+1)$'s by $(n-k)$'s so that the $\max(S_{a_1+\dots+a_{\ell-k+1}}) \le n-k$ restriction is satisfied. (If some set contains both $n-k$ and $n-k+1$, we simply delete $n-k+1$ from that set.) Changing the $(n-k+1)$'s into $(n-k)$'s will not cause any of the strict increase conditions to be violated because of the assumption that $a_{\ell-k}\le n-k-1$, since then if the first of the $a_{\ell-k+1}$ sets is equal to $\{n-k+1\}$ and gets changed to $\{n-k\}$, it will still be strictly greater than the final set in the $a_{\ell-k}$ group.
    
    Note that this corresponding term from before applying $\p_{n-k}$ must contain $x_{n-k}$ and cannot contain $x_{n-k+1}$, so let $1\le b \le a_{\ell-k+1}$ be the exponent on $x_{n-k}$. When applying $\p_{n-k}$, all other variables besides $x_{n-k}$ and $x_{n-k+1}$ are treated like constants, so we may ignore them and just consider the action of $\p_{n-k}$ on $x_{n-k}^b$. We get 
    \begin{align*}
        \p_{n-k}(x_{n-k}^b) &= \frac{x_{n-k}^{b+1}-x_{n-k+1}^{b+1} + \beta x_{n-k}x_{n-k+1}(x_{n-k}^b - x_{n-k+1}^b)}{x_{n-k}-x_{n-k+1}} \\
        &= \sum_{c+d = b}x_{n-k}^c x_{n-k+1}^d + \beta x_{n-k}x_{n-k+1}\sum_{c+d = b-1}x_{n-k}^c x_{n-k+1}^d.
    \end{align*}
    The terms we get from the first sum correspond to transformations of the form $$\underbrace{\dots n-k|\dots|n-k|}_{b} \ \ \to \ \ \underbrace{\dots n-k|\dots|n-k}_c|\underbrace{n-k+1|\dots|n-k+1|}_d,$$ where $c+d=b$ and the $b$ relevant sets are the last $b$ sets in the $a_{\ell-k+1}$ group. The terms we get from the second sum correspond to transformations of the form $$\underbrace{\dots n-k|\dots|n-k|}_{b} \ \ \to \ \ \underbrace{\dots n-k|\dots|n-k}_c|n-k,n-k+1|\underbrace{n-k+1|\dots|n-k+1|}_d,$$ where again the $b$ starting sets are at the end of the $a_{\ell-k+1}$ group, but this time $c+d = b-1$ and we pick up an additional $\beta$ because one of the $x_{n-k}$'s turns into both an $x_{n-k}$ and an $x_{n-k+1}$. These are exactly the terms we wanted to gain from applying $\p_{n-k}$.
    
    It remains to check that the other terms that exist after applying $\p_{n-k}$ are precisely the same as the remaining terms that were present before applying $\p_{n-k}$. Such a term must either contain both an $n-k$ and an $n-k+1$, or it could contain just one of $n-k$ or $n-k+1$. In the former case where both $n-k$ and $n-k+1$ are used, we get a monomial containing both $x_{n-k}$ and $x_{n-k+1}$, which is fixed by $\p_{n-k}$ by the calculation in (\ref{eqn:pi_fixed_things}). 
    
    In the latter case where we have just one of $n-k$ or $n-k+1$, it is never possible for an $n-k$ or $n-k+1$ to be in the $a_{\ell-k+2}$ group or later, as we noted above. It is also not possible to have an $n-k+1$ in the $a_{\ell-k+1}$ group before applying $\p_{n-k}$ y the inductive hypothesis, and we have already considered all cases where there is an $n-k$ in the $a_{\ell-k+1}$ group. Thus, the only possibility is that the $n-k$ or $n-k+1$ is in the $a_{\ell-k}$ group or earlier. But in that case, there is a corresponding term that is the same but all the $(n-k)$'s replaced with $(n-k+1)$'s or vice versa, since swapping all the $(n-k)$'s for $(n-k+1)$'s in the $a_{\ell-k}$ group or earlier will not cause any conditions to be violated. Thus, factoring out all the variables except $x_{n-k}$ and $x_{n-k+1}$, we have some pair of terms of the form $x_{n-k}^b + x_{n-k+1}^b$. Since $\til{s}_{n-k}$ swaps $x_{n-k}^b$ with $x_{n-k+1}^b$, it fixes their sum, so $\p_{n-k}$ also fixes their sum by the calculation in (\ref{eqn:pi_fixed_things}). Thus, $\p_{n-k}$ introduces exactly the terms we want and fixes all other terms, completing the proof of Lemma \ref{lem}.
\end{proof}

To finish the proof of (\ref{eqn:F}), it remains to check that the final operators $\p_1\dots\p_{n-\ell-1}$ fix our polynomial, since Lemma \ref{lem} implies that after applying $\p_{n-\ell}\dots \p_{n-1}$ we already have the polynomial we want. For each $1\le i\le n-\ell-1$, terms involving both or neither of $x_i$ and $x_{i+1}$ are fixed by the calculation in (\ref{eqn:pi_fixed_things}). For each term involving just one of $x_i$ or $x_{i+1}$, there is always a corresponding term involving just the other, since we can swap all the $x_i$'s for $x_{i+1}$'s or vice versa without violating any conditions. Then after factoring out the other variables we have a sum of two terms $x_i^b + x_{i+1}^b$ such that the sum is fixed by $\til{s}_i$, so the sum is also fixed by $\p_i$, as needed. By induction, this completes the proof that $\p_{w_0}(x^a) = \ol{F}_a$. \qed

\subsection{Proof that \texorpdfstring{$\ol{\til{\pi}}_{w_a^{-1}}(x^{\tn{flat}(a)}) = \ol{\mf{F}}_a$}{second formula}}

We use induction on the length of $w_a$. For the base case, if $w_a=\tn{id}$ has length 0, then $a = \tn{flat}(a)$ and it is not possible to move any entries of $a$ to the left, so $\F_a = x^a = \p_{w_a^{-1}}(x^a) = \p_{w_a^{-1}}(x^{\tn{flat}(a)}).$

For the inductive step, note that it is always possible to turn $\tn{flat}(a)$ into $a$ by a sequence of adjacent transpositions that swap a nonzero number with a 0 immediately to its right. (For instance, we can use these transposition to first move the rightmost nonzero part from its position in $\tn{flat}(a)$ to its position in $a$, then move the next-to-rightmost nonzero part to its position in $a$, and so on.)

Thus, we can let $a = \sigma_i(a'),$ where $\sigma_i$ is such a transposition that swaps a nonzero entry $a_i'$ with the entry $a'_{i+1}=0$ to give $a_i=0$, $a_{i+1}=a_i'$, and $a_j=a'_j$ for $j\ne i,i+1$. Then $w_a^{-1} = \sigma_iw_{a'}^{-1}$ and $\tn{flat}(a) = \tn{flat}(a')$, so we may assume by induction that $\p_{w_{a'}^{-1}}(x^{\tn{flat}(a)}) = \F_{a'},$ and it suffices to show that $\p_i(\F_{a'})=\F_a.$

All glides of $a'$ are also glides of $a$, so $\F_a$ contains all the same terms as $\F_{a'}$, plus additional terms for the glides $b$ of $a$ that are not glides of $a'$. If $a_{i+1}$ is the $j$th nonzero entry of $a$, then those additional terms correspond precisely to the glides $b$ such that the $j$th block of $b$ ends at position $i+1$ (the rightmost position at which it is allowed to end) and $b_{i+1}\ne0$. Any such glide $b$ can be uniquely obtained from a glide $b'$ of $a'$ with $b'_i\ne 0$ and $b'_{i+1}=0$ by a local move of one of the following forms:
\begin{align*}
    &b'_i 0 \to b_ib_{i+1} \ \tn{ or } \ \boldsymbol{\color{red}\ol{b'}_i}0 \to \boldsymbol{\color{red}\ol{b}_i}b_{i+1} &\tn{ with } \ b_i + b_{i+1} = b'_i, \\
    &b'_i 0 \to b_i\boldsymbol{\color{red}\ol{b}_{i+1}} \ \tn{ or } \ \boldsymbol{\color{red}\ol{b'}_i}0 \to \boldsymbol{\color{red}\ol{b}_i}\boldsymbol{\color{red}\ol{b}_{i+1}} \ &\tn{ with } \ b_i + b_{i+1} = b'_i+1, \tn{ and } b_i, b_{i+1}\ne 0
\end{align*}
Thus, we need to show that applying $\p_i$ to $\F_{a'}$ keeps all terms that were there already and adds additional terms corresponding precisely to glides that can be obtained from the moves above.

As noted in (\ref{eqn:pi_fixed_things}), $\p_i$ fixes all monomials involving both or neither of $x_i$ and $x_{i+1}$. This is what we want because for any such $x^{b'}$ involving both or neither of $x_i$ and $x_{i+1}$, $b'$ is a glide of $a'$ if and only if it is a glide of $a$, so the monomial $x^{b'}$ should occur in both or neither of $\F_{a'}$ and $\F_a$. Thus, it remains to consider the action of $\p_i$ on monomials $x^{b'}$ where exactly one of $b'_i$ and $b'_{i+1}$ is nonzero. 

For a glide $b'$ of $a'$ with $b'_i = 0$ and $b'_{i+1}\ne 0$, it must be the case that swapping $b_i'$ and $b'_{i+1}$ also gives a glide of $a'$, so we actually have a pair of monomials $x^{b'}$ and $s_i(x^{b'})=\til{s}_i(x^{b'})$ that both occur in $\F_{a'}$. Then $\til{s}_i$ swaps those two monomials and hence fixes their sum, so $\p_i$ fixes their sum by (\ref{eqn:pi_fixed_things}). That is what we need, because if those two monomials both correspond to glides of $a'$, then they also correspond to glides of $a$, so they should both occur in $\F_{a'}$.

The remaining glides $b'$ of $a'$ to consider are ones with $b'_i\ne 0$ and $b'_{i+1}=0$, such that applying $\sigma_i$ does not give another glide of $a'$, so the monomial $x^{b'}$ occurs in $\F_{a'}$ but $s_i(x^{b'})$ does not. Then we get
\begin{equation}\label{eqn:pi(x^b)}
    \p_i(x_i^{b_i'}) = \frac{x_i^{b_i'+1}-x_{i+1}^{b_i'} + \beta x_i x_{i+1}(x_i^{b_i'} - x_{i+1}^{b_i'})}{x_i-x_{i+1}} = \sum_{b_i + b_{i+1} = b_i'} x_i^{b_i}x_{i+1}^{b_{i+1}} + \beta \sum_{\substack{b_i + b_{i+1} = b_i' + 1, \\ b_i,b_{i+1}\ge 1}} x_i^{b_i} x_{i+1}^{b_{i+1}}.
\end{equation}
These terms correspond precisely to the glides of $a$ obtained from $b'$ by the moves listed above, as needed. \qed

\subsection{Proof that \texorpdfstring{$\ol{\til{\theta}}_{w_a^{-1}}(x^{\tn{flat}(a)}) = \ol{\mf{P}}_a$}{third formula}}

We use induction on the length of $a$ with the same setup as before. The base case holds because if $w_a = \tn{id}$, then $\P_a = x^a = \thet_{w_a^{-1}}(x^a) = \thet_{w_a^{-1}}(x^{\tn{flat}(a)}).$ Now assume $a = \sigma_i(a')$ with $w_a^{-1} = \sigma_i w_{a'}^{-1}$ such that $\sigma_i$ swaps $a'_i\ne 0$ with $a'_{i+1}=0$. We can assume by induction that $\P_{a'} = \thet_{w_{a'}^{-1}}(x^{\tn{flat}(a')}) = \thet_{w_{a'}^{-1}}(x^{\tn{flat}(a)}),$ so it suffices to show that $\P_a = \thet_i(\P_{a'}).$

Since $\P_{a'}$ is the generating series for mesonic glides of $a'$ and $a'_i\ne 0$, it follows from mesonic glide requirement that each block of $b'$ end precisely at the position of the corresponding nonzero entry of $a'$ and have a nonzero entry in that position that all terms in $\P_{a'}$ have a nonzero exponent on $x_i$. similarly, all terms of $\P_a$ should have a nonzero exponent on $x_{i+1}$. Furthermore, unlike for $\F_{a'}$ and $\F_a$, mesonic glides of $a'$ will never also be mesonic glides of $a$, but each mesonic glide of $a$ can be obtained from a mesonic glide of $a'$ by one of the following local moves:
\begin{align*}
    &b'_i 0 \to b_ib_{i+1} \ \tn{ or } \ \boldsymbol{\color{red}\ol{b'}_i}0 \to \boldsymbol{\color{red}\ol{b}_i}b_{i+1} &\tn{ with } \ b_i + b_{i+1} = b'_i\tn{ and }b_{i+1}\ne 0, \\
    &b'_i 0 \to b_i\boldsymbol{\color{red}\ol{b}_{i+1}} \ \tn{ or } \ \boldsymbol{\color{red}\ol{b'}_i}0 \to \boldsymbol{\color{red}\ol{b}_i}\boldsymbol{\color{red}\ol{b}_{i+1}} \ &\tn{ with } \ b_i + b_{i+1} = b'_i+1, \tn{ and } b_i, b_{i+1}\ne 0.
\end{align*}

Mesonic glides $b'$ of $a'$ with $b_i'\ne0$ and $b_{i+1}'\ne 0$ cannot be turned into a mesonic glide of $a$ by one of these moves, since $b'_{i+1}$ belongs to the next block of $b'$ corresponding to the next nonzero entry of $a'$, so the $x^{b'}$ terms with $b'_i\ne0$ and $b'_{i+1}\ne 0$ should go away when we apply $\thet_i$ to turn $\P_{a'}$ to $\P_a$. To check that this indeed happens, note that all such monomials involving both $x_i$ and $x_{i+1}$ are fixed by $\til{s}_i$ and so are sent to 0 by $\del_i$, hence they are also sent to 0 by $\thet_i = x_{i+1}(1+\beta x_i)\del_i$.

For mesonic glides $b'$ of $a'$ with $b_i'\ne 0$ and $b'_{i+1}=0$, we can factor out the part not involving $x_i$ and then find the action of $\thet_i$ on $x_i^{b_i'}$ by simply subtracting the $x^{b_i'}$ term on the right side of (\ref{eqn:pi(x^b)}) to get $$\thet_i(x_i^{b_i'}) = (\p_i-1)(x_i^{b_i'}) = \sum_{\substack{b_i + b_{i+1} = b_i',\\ b_{i+1}\ge 1}} x_i^{b_i}x_{i+1}^{b_{i+1}} + \beta \sum_{\substack{b_i + b_{i+1} = b_i' + 1, \\ b_i,b_{i+1}\ge 1}} x_i^{b_i} x_{i+1}^{b_{i+1}}.$$ These terms correspond precisely to glides that can be obtained from $b'$ by the moves above, since the only difference from the $\p_i$ case is the requirement that $b_{i+1}\ne 0$ for the moves on the first line. Thus, applying $\thet_i$ to $\P_{a'}$ gives exactly the terms of $\P_a$, as needed. \qed

\section*{Acknowledgments}

I'm thankful to Oliver Pechenik for suggesting this project.

\printbibliography

\end{document}